\documentclass[11pt]{amsart}

\newtheorem{theorem}{Theorem}[section]
\newtheorem{lemma}[theorem]{Lemma}

\theoremstyle{definition}

\newtheorem{proposition}[theorem]{Proposition}
\newtheorem{corollary}[theorem]{Corollary}
\newtheorem{remark}[theorem]{Remark}
\newtheorem{conjecture}[theorem]{Conjecture}
\theoremstyle{remark}

\newcommand{\be}{\begin{equation}}
\newcommand{\ee}{\end{equation}}

\numberwithin{equation}{section}

\begin{document}

\title{Some Remarks on Circle Action on Manifolds}

\author{Ping Li}
\address{Department of Mathematics, Tongji University, Shanghai 200092, China}
\email{pingli@tongji.edu.cn}
\thanks{The first author is supported by Program for Young Excellent
Talents in Tongji University.}

\author{Kefeng Liu} \address{Department of Mathematics, University of California at Los Angeles, Los Angeles, CA 90095, USA and Center of Mathematical Science,
Zhejiang University, 310027, China} \email{liu@math.ucla.edu}




\begin{abstract}
This paper contains several results concerning circle action on
almost-complex and smooth manifolds. More precisely, we show that,
for an almost-complex manifold $M^{2mn}$(resp. a smooth manifold
$N^{4mn}$), if there exists a partition
$\lambda=(\lambda_{1},\cdots,\lambda_{u})$ of weight $m$ such that
the Chern number $(c_{\lambda_{1}}\cdots c_{\lambda_{u}})^{n}[M]$
(resp. Pontrjagin number $(p_{\lambda_{1}}\cdots
p_{\lambda_{u}})^{n}[N]$) is nonzero, then \emph{any} circle action
on $M^{2mn}$ (resp. $N^{4mn}$) has at least $n+1$ fixed points. When
an even-dimensional smooth manifold $N^{2n}$ admits a semi-free
action with isolated fixed points, we show that $N^{2n}$ bounds,
which generalizes a well-known fact in the free case. We also
provide a topological obstruction, in terms of the first Chern
class, to the existence of semi-free circle action with
\emph{nonempty} isolated fixed points on almost-complex manifolds.
The main ingredients of our proofs are Bott's residue formula and
rigidity theorem.
\end{abstract}

\maketitle

\section{Introduction and main results}
Unless otherwise stated, all the manifolds (smooth or
almost-complex) mentioned in the paper are closed, connected and
oriented. For almost-complex manifolds, we take the canonical
orientations induced from the almost-complex structures. We denote
by superscripts the corresponding \emph{real} dimensions of such
manifolds. When $M$ is a smooth (resp. almost-complex) manifold, we
say $M$ has an $S^{1}$-action if $M$ admits a circle action which
preserves the smooth (resp. almost-complex) structure.

Given a manifold $M$ and an $S^{1}$-action, the study of the fixed
point set $M^{S^{1}}$ is an important topic in geometry and
topology. In (\cite{Ko}, p.338), Kosniowski proposed the following
conjecture, which relates the number of fixed points to the
dimension of the manifold.

\begin{conjecture}[Kosniowski]Suppose that $M^{2n}$ is a unitary
$S^{1}$-manifold with isolated fixed points. If $M$ is not a
boundary then this action has at least $[\frac{n}{2}]+1$ fixed
points.
\end{conjecture}
\begin{remark}A \emph{weakly almost-complex structure} on a manifold $M^{2n}$ is determined by a complex structure in the vector bundle $\tau(M^{2n})\oplus\mathbb{R}^{2k}$ for some $k$, where
$\tau(M^{2n})$ is the tangent bundle of $M^{2n}$ and
$\mathbb{R}^{2k}$ denotes a trivial real $2k$-dimensional vector
bundle over $M^{2n}$. A unitary $S^{1}$-manifold means that $M^{2n}$
has a weakly almost-complex structure and $S^{1}$ acts on $M^{2n}$
preserving this structure.
\end{remark}

Recently, Pelayo and Tolman showed that (\cite{PT}, Theorem 1), if a
symplectic manifold $(M^{2n}, \omega)$ has a symplectic
$S^{1}$-action and the weights induced from the isotropy
representations on the fixed points satisfy some subtle condition,
then this action has at least $n+1$ fixed points.

\begin{remark}If a
symplectic manifold $(M^{2n}, \omega)$ has an \emph{Hamiltonian}
$S^{1}$-action, then the fact that this action must have at least
$n+1$ fixed points is quite well-known. The reason is that the fixed
points are exactly the critical points of the corresponding momentum
map (a perfect Morse-Bott function) and the even-dimensional Betti
numbers of $M$ are all positive. The conclusion then follows from
the Morse inequality. This reason has been explained in details in
the Introduction of \cite{PT}.
\end{remark}

We recall that a \emph{partition} is a finite sequence
$\lambda=(\lambda_{1},\cdots,\lambda_{u})$ of unordered positive
integers. We call $\sum_{i=1}^{u}\lambda_{i}$ the \emph{weight} of
this partition $\lambda$.

Inspired by the techniques from \cite{PT}, we will show the
following theorem in Section 3, which is our first main result.

\begin{theorem}\label{maintheorem1}
\begin{enumerate}
\item
Suppose $M^{2mn}$ is an almost-complex manifold. If there exists a
partition $\lambda=(\lambda_{1},\cdots,\lambda_{u})$ of weight $m$
such that the corresponding Chern number $(c_{\lambda_{1}}\cdots
c_{\lambda_{u}})^{n}[M]$ is nonzero, then \emph{any} $S^{1}$-action
on $M$ must have at least $n+1$ fixed points.

\item
Suppose $N^{4mn}$ is a smooth manifold. If there exists a partition
$\lambda=(\lambda_{1},\cdots,\lambda_{u})$ of weight $m$ such that
the corresponding Pontrjagin number $(p_{\lambda_{1}}\cdots
p_{\lambda_{u}})^{n}[M]$ is nonzero, then \emph{any} $S^{1}$-action
on $N$ must have at least $n+1$ fixed points.

\end{enumerate}

\end{theorem}

\begin{corollary}
\begin{enumerate}
\item If the Chern number $c_{m}^{n}[M]$ is nonzero, then for any $S^{1}$-action on almost-complex manifold $M^{2mn},$ it has at least $n+1$ fixed points.
In particular, if $c_{1}^{n}[M]$ is nonzero, then any $S^{1}$-action
on almost-complex manifold $M^{2n}$ must have at least $n+1$ fixed
points.

\item If the Pontrjagin number $p_{m}^{n}[N]$ is nonzero, then for any $S^{1}$-action on
smooth manifold $N^{4mn}$, it has at least $n+1$ fixed points. In
particular, if $p_{1}^{n}[N]$ is nonzero, then any $S^{1}$-action on
$N^{4n}$ must have at least $n+1$ fixed points.
\end{enumerate}
\end{corollary}

It is a well-known fact that, if a smooth manifold $N^{n}$ has a
free $S^{1}$-action, then $N^{n}$ bounds, i.e., $N^{n}$ can be
realized as the oriented boundary of some smooth, oriented,
$(n+1)$-dimensional manifold with boundary. Using the language of
cobordism theory, $[N^{n}]=0\in \Omega_{\ast}^{SO}$, where
$\Omega_{\ast}^{SO}$ is the oriented cobordism ring. In particular,
all the Pontrjagin numbers and Stiefel-Whitney numbers vanish. This
well-known fact is not difficult to prove: $N^{n}$ is the total
space of the principal $S^{1}$-bundle over the quotient manifold
$N^{n}/S^{1}$, of which the structure group is $S^{1}=SO(2)$. Then
we can extend the action of $SO(2)$ to the $2$-disk $D^{2}$ to get
the associated $D^{2}$-bundle $N^{n}\times_{SO(2)}D^{2}$, of which
the boundary is exactly $N^{n}$.

We recall that a circle action is called \emph{semi-free} if it is
free outside the fixed point set or equivalently, the isotropic
subgroup of any point is either trivial or the whole circle. In
(\cite{PW}, Theorem 1.1), the authors showed that the Pontrjagin
numbers of manifolds admitting a semi-free action with isolated
fixed points are all zero. Our following result is a generalization
of both the well-known fact mentioned above and (\cite{PW}, Theorem
1.1).
\begin{theorem}\label{maintheorem2}
If an even-dimensional smooth manifold $N^{2n}$ admits a semi-free
$S^{1}$-action with isolated fixed points, then $N^{2n}$ bounds. In
particular,  all the Pontrjagin numbers and Stiefel-Whitney numbers
vanish.
\end{theorem}

\begin{remark}
In this theorem, the semi-free hypothesis is essential. For example,
we can look at the $S^{1}$-action on complex projective plane
$\mathbb{C}P^2$ given by
$$[z_{0}:z_{1}:z_{2}]\rightarrow [z_{0}:t\cdot z_{1}:t^{2}\cdot z_{2}]$$
for $t\in S^{1}$, where $[z_{0}:z_{1}:z_{2}]$ denotes the
homogeneous coordinates of $\mathbb{C}P^2$.

The fixed points of this action are $[1:0:0]$, $[0:1:0]$ and
$[0:0:1]$. The isotropy subgroup of the non-fixed point $[1:0:1]$ is
$\{1,-1\}=\mathbb{Z}_{2}$, which is nontrivial. Hence this action is
not semi-free.
\end{remark}
This result tells us that Pontrjagin numbers and Stiefel-Whitney
numbers are numerical obstructions to the existence of semi-free
actions with isolated fixed points on smooth manifolds. However, in
contrast to the smooth case, when an almost-complex manifold
$(M^{2n},J)$ has a semi-free $S^{1}$-action with \emph{nonempty}
isolated fixed points, the Chern numbers of $(M^{2n},J)$ don't
vanish (see Lemma \ref{c_{1}c_{n-1}}). One may ask, in the
almost-complex case, whether there still exist some topological
obstructions to the existence of semi-free $S^{1}$-actions. We know
that the first Chern class plays an important role in complex
(almost-complex) manifolds. The following result provides such an
obstruction to the existence of semi-free $S^{1}$-action on
almost-complex manifolds.
\begin{theorem}\label{maintheorem3}
Let $(M^{2n},J)$ be an almost-complex manifold admitting a semi-free
$S^{1}$-action with nonempty isolated fixed points. Then the first
Chern class $c_{1}(M)\in H^{2}(M;\mathbb{Z})$ is either primitive or
twice a primitive element.
\end{theorem}
\begin{remark}For almost-complex manifolds, in the semi-free case with isolated fixed points, the only known examples are
$(\mathbb{C}P^{1})^{n}$, equipped with the diagonal $S^{1} $-action.
Note that these examples are even Hamiltonian $S^{1}$-actions on
symplectic manifolds. In fact, Hattori showed that (\cite{Ha}), if a
symplectic manifold $(M^{2n},\omega)$ admits a Hamiltonian,
semi-free $S^{1}$-action with isolated fixed points, then the
cohomology ring and the Chern classes of $(M^{2n},\omega)$ are the
same as $(\mathbb{C}P^1)^{n}$. Recently, Gonzalez showed that
(\cite{Go1}) such $(M^{2n},\omega)$ has the same quantum cohomology
ring as $(\mathbb{C}P^1)^{n}$. Moreover, if $n=3$, Gonzalez showed
that (\cite{Go2}) $(M^{6},\omega)$ is equivariantly symplectomorphic
to $(\mathbb{C}P^1)^{3}$. While in the almost-complex case, much
less is known. It would be very interesting to find some more
topological obstructions.
\end{remark}
In Section 2, we will review the Bott's residue formula and prove a
rigidity proposition. Then in Section 3, the three subsections will
be devoted to the proofs of Theorems \ref{maintheorem1},
\ref{maintheorem2} and \ref{maintheorem3} respectively.

\section{Preliminaries}
\subsection{Bott's residue formula}
\subsubsection{almost-complex case}
Let $(M^{2n},J)$ be an almost-complex manifold with a circle action
with isolated fixed points, say $\{P_{1},\cdots,P_{r}\}$. In each
fixed point $P_{i}$, there are well-defined $n$ integer weights
$k_{1}^{(i)},\cdots, k_{n}^{(i)}$ (not necessarily distinct) induced
from the isotropy representation of this $S^{1}$-action on the
holomorphic tangent space $T_{p_{i}}M$ in the sense of $J$. Note
that these $k^{(i)}_{1},\cdots,k^{(i)}_{n}$ are \emph{nonzero} as
the fixed points are isolated. Let $f(x_{1},\cdots,x_{n})$ be a
symmetric polynomial in the variables $x_{1},\cdots,x_{n}$. Then
$f(x_{1},\cdots,x_{n})$ can be written in an essentially unique way
in terms of the elementary symmetric polynomials
$\widetilde{f}(e_{1},\cdots,e_{n})$, where
$e_{i}=e_{i}(x_{1},\cdots,x_{n})$ is the $i$-th elementary symmetric
polynomial of $x_{1},\cdots,x_{n}$.

Now we can state a version of the Bott residue formula \big(cf.
\cite{Bo} or (\cite{AS}, p.598)\big) which reduces the computations
of Chern numbers of $(M^{2n},J)$ to $\{k_{j}^{(i)}\}$, as follows.
\begin{theorem}[Bott residue formula]
With above notations understood and moreover suppose the degree of
$f(x_{1},\cdots,x_{n})$ is not greater than $n$
($\textrm{deg}(x_{i})=1$). Then\be\label{BRF1}
\sum_{i=1}^{r}\frac{f(k_{1}^{(i)},\cdots,k_{n}^{(i)})}{\prod_{j=1}^{n}k_{j}^{(i)}}=\widetilde{f}(c_{1},\cdots,c_{n})\cdot[M],\ee
where $c_{i}$ is the $i$-th Chern class of $(M^{2n},J)$ and $[M]$ is
the fundamental class of $M$ induced from $J$.\end{theorem}
\begin{remark}\label{AMRMK}
If the degree of $f(x_{1},\cdots,x_{n})$ is less than $n$, then the
left-hand side of (\ref{BRF1}) vanishes. If the action has no fixed
points (though not necessarily free) it follows that all Chern
numbers are zero.
\end{remark}
\subsubsection{smooth case}
Let $N^{2n}$ be a smooth manifold with a circle action with isolated
fixed points, say $\{P_{1},\cdots,P_{r}\}$. In each fixed point
$P_{i}$, the tangent space $T_{p_{i}}N$ splits as an $S^{1}$-module
induced from the isotropy representation as
follows
$$T_{p_{i}}N=\bigoplus_{j=1}^{n}V^{(i)}_{j},$$
where each $V_{j}^{(i)}$ is a real $2$-plane. We choose an
isomorphism of $\mathbb{C}$ with $V_{j}^{(i)}$ relative to which the
representation of $S^{1}$ on $V_{j}^{(i)}$ is given by
$e^{\sqrt{-1}\theta}\mapsto e^{\sqrt{-1}k_{j}^{(i)}\theta}$ with
$k_{j}^{(i)}\in\mathbb{Z}-\{0\}$. We can assume the rotation numbers
$k^{(i)}_{1},\cdots,k^{(i)}_{n}$ be chosen in such a way that the
usual orientations on the summands $V^{(i)}_{j}\cong\mathbb{C}$
induce the given orientation on $T_{p_{i}}N$. Note that these
$k^{(i)}_{1},\cdots,k^{(i)}_{n}$ are uniquely defined up to even
number of sign changes. In particular, their product
$\prod_{j=1}^{n}k_{j}^{(i)}$ is well-defined.

Let $f(x_{1}^{2},\cdots,x_{n}^{2})$ be a symmetric polynomial in the
variables $x_{1}^{2},\cdots,x_{n}^{2}$. Let
$\sigma_{i}=\sigma_{i}(x_{1}^{2},\cdots,x_{n}^{2})$ be the $i$-th
elementary symmetric polynomial in the variables
$x_{1}^{2},\cdots,x_{n}^{2}$. Then $f(x_{1}^{2},\cdots,x_{n}^{2})$
can be written in an essentially unique way in terms of
$\sigma_{1},\cdots,\sigma_{n}$, say
$\widetilde{f}(\sigma_{1},\cdots,\sigma_{n})$. Then we have
\begin{theorem}[Bott residue formula]
With above notations understood and moreover suppose the degree of
$f(x_{1}^{2},\cdots,x_{n}^{2})$ is not greater than $n$
(deg$(x_{i})=1$). Then\be\label{BRF2}
\sum_{i=1}^{r}\frac{f((k_{1}^{(i)})^{2},\cdots,(k_{n}^{(i)})^{2})}{\prod_{j=1}^{n}k_{j}^{(i)}}=\widetilde{f}(p_{1},\cdots,p_{n})\cdot[N],\ee
where $p_{i}$ is the $i$-th Pontrjagin class of $N$ and $[N]$ is the
fundamental class of $N$ determined by the orientation.
\end{theorem}
\begin{remark}
Since deg$(f(x_{1}^{2},\cdots,x_{n}^{2}))\leq n$, what possible
appear in $\widetilde{f}(p_{1},\cdots,p_{n})$ are
$p_{1},\cdots,p_{[\frac{n}{2}]}$.
$\widetilde{f}(p_{1},\cdots,p_{n})\cdot[N]$ is nonzero only if $n$
is even. If $\textrm{deg}(f(x_{1}^{2},\cdots,x_{n}^{2}))<n$, then
the left-hand side of (\ref{BRF2}) vanishes. If the action has no
fixed points (though not necessarily free), it follows that all
Pontrjagin numbers are zero.
\end{remark}

\subsection{A rigidity result}
In this subsection we want to prove a special rigidity result for
circle actions on almost-complex manifolds with isolated fixed
points. For more details on the rigidity of elliptic complexes, see
\cite{Liu1} and
\cite{Liu2}.

For more details on the following paragraphs in this subsection, we
recommend the readers the references \cite{Hi2} or Appendix III of
\cite{HBJ}. Let $(M^{2n},J)$ be an almost-complex manifold with
first Chern class $c_{1}\in H^{2}(M;\mathbb{Z})$ divisible by a
positive integer $d>1$. Then there exists a complex line bundle $L$
over $M$ such that $L^{\otimes d}=K$, where $K$ is the canonical
complex line bundle of $M$ in the sense of $J$. Let $\chi(M,L)$ be
the complex genus (\cite{HBJ}, p.18) corresponding to the
characteristic power series \be\label{CPS}\frac{x}{1-e^{-x}}\cdot
e^{-\frac{x}{d}}.\ee Note that the Todd genus corresponds to the
characteristic power series $\frac{x}{1-e^{-x}}$, which means
\be\label{relation}\chi(M,L)=\big(\textrm{ch}(L)\cdot\textrm{td}(M)\big)[M].\ee
Here $\textrm{ch}(L)$ is the Chern character of $L$ and
$\textrm{td}(M)$ is the Todd class of $M$. By (\ref{relation}),
$\chi(M,L)$ can be realized as the index of a suitable elliptic
operator twisted by $L$ (cf. \cite{HBJ}, p.167).

Now suppose we have an $S^{1}$-action on $(M^{2n},J)$. Consider the
$d$-fold covering $S^{1}\rightarrow S^{1}$ with
$\mu\mapsto\lambda=\mu^{d}$. Then $\mu$ acts on $M$ and $K$ through
$\lambda$. This action can be lifted to $L$. Then for any $g\in
S^{1}$, we can define the equivariant index $\chi(g;M,L)$, which is
a finite Laurent series in $g$.

Now suppose this circle action on $M$ has isolated fixed points.
Using the notations in Section $2.1.1$, we have
\begin{proposition}\label{rigidity proposition}Suppose the first Chern class of $M$ is divisible
by $d>1$. Then the rational function
$$\sum_{i=1}^{r}\frac{g^{\frac{\sum_{j=1}^{n}k_{j}^{(i)}}{d}}}{\prod_{j=1}^{n}(1-g^{k_{j}^{(i)}})}$$
is identically equal to $0$, where $g$ is an indeterminate.
\end{proposition}
\begin{proof}
Suppose $g\in S^{1}$ is a topological generator. Then the fixed
points of the action $g$ are exactly $\{P_{1},\cdots,P_{r}\}$. Note
that the characteristic power series corresponding to $\chi(M,L)$ is
(\ref{CPS}), then the Lefschetz fixed point formula of
Atiyah-Bott-Segal-Singer (\cite{AS}, p.562) tells us that
$$\chi(g;M,L)=\sum_{i=1}^{r}\prod_{j=1}^{n}\frac{g^{\frac{k_{j}^{(i)}}{d}}}{1-g^{k_{j}^{(i)}}}.$$
The rigidity result of almost-complex manifolds on the level of $d$
\big(cf. (p.43 and p.58 of \cite{Hi2}) or (p.173 and p.183 of
\cite{HBJ})\big) tells us that, for any topological generator $g\in
S^{1}$,
$$\chi(g;M,L)\equiv\chi(M,L).$$ Since the topological generators in
$S^{1}$ are dense, we have an identity
$$\chi(M,L)\equiv\sum_{i=1}^{r}\prod_{j=1}^{n}\frac{g^{\frac{k_{j}^{(i)}}{d}}}{1-g^{k_{j}^{(i)}}}$$
 for any indeterminate $g$.

For any $k_{j}^{(i)}\in\mathbb{Z}-\{0\},$ we have
$$\lim_{g\rightarrow\infty}\frac{g^{\frac{k_{j}^{(i)}}{d}}}{1-g^{k_{j}^{(i)}}}=0.$$
 Therefore,

$$\sum_{i=1}^{r}\prod_{j=1}^{n}\frac{g^{\frac{k_{j}^{(i)}}{d}}}{1-g^{k_{j}^{(i)}}}\equiv 0$$
for any indeterminate $g$, which completes the proof.
\end{proof}

\begin{remark}
The appendix III of \cite{HBJ} is only a copy of \cite{Hi2}.
Although the results in \cite{Hi2} are formulated for complex
manifolds, the tools and methods can also been applied to
almost-complex manifolds. Hence the results in \cite{Hi2} are also
valid for \emph{almost-complex manifolds}, which have been pointed
out by Hirzebruch himself in his original paper (p.38 of \cite{Hi2}
or p.170 of \cite{HBJ}).
\end{remark}
\section{Proof of main results}
\subsection{Proof of Theorem \ref{maintheorem1}}
Now suppose $(M^{2mn},J)$ is an almost-complex manifold with some
Chern number $(c_{\lambda_{1}}\cdots c_{\lambda_{u}})^{n}[M]\neq 0$.
Note that any $S^{1}$-action on $M$ must have at least one fixed
point, otherwise all the Chern numbers of $M$ vanish by Remark
\ref{AMRMK}. If the fixed point set of the $S^{1}$-action is not
isolated, then at least one connected component is a submanifold of
positive dimension. In this case there are infinitely many fixed
points. To complete the proof of the first part of Theorem
\ref{maintheorem1}, it suffices to consider the $S^{1}$-actions with
nonempty isolated fixed points.

Like the notations in Section $2.1.1$, we assume the isolated fixed
points are $\{P_{1},\cdots,P_{r}\}$. In each fixed point $P_{i}$ we
have $mn$ integer weights $k^{(i)}_{1}, k^{(i)}_{2}\cdots,
k^{(i)}_{mn}$. Given any partition
$\lambda=(\lambda_{1},\cdots,\lambda_{u})$ of weight $m$, We define
$$c_{\lambda}(i):=\prod_{t=1}^{u}\big(\sum_{1\leq j_{1}<j_{2}<\cdots<j_{\lambda_{t}}\leq mn}k^{(i)}_{j_1}k^{(i)}_{j_2}\cdots
k^{(i)}_{j_{\lambda_{t}}}\big).$$ Let \be\label{Chern
map}\{c_{\lambda}(i)~|~1\leq i\leq
r\}=\{s_{1},\cdots,s_{l}\}\subset\mathbb{Z}\ee and define
$$A_{t}:=\sum_{\substack{1\leq i\leq r \\ c_{\lambda}(i)=s_{t}}}\frac{1}{\prod_{j=1}^{mn}k_{j}^{(i)}},\qquad 1\leq t\leq
l.$$
\begin{lemma}\label{main1lemma1}
If the Chern number $(c_{\lambda_{1}}\cdots
c_{\lambda_{u}})^{n}[M]\neq 0$, then at least one of $A_{t}$ is
nonzero.
\end{lemma}

 \begin{proof} Suppose $A_{t}=0$ for
all $t=1,\cdots,l$. Then Bott residue formula (\ref{BRF1}) tells us
$$(c_{\lambda_{1}}\cdots
c_{\lambda_{u}})^{n}[M]=\sum_{t=1}^{l}(s_{t})^{n}\cdot A_{t}=0.$$
\end{proof}

The following lemma is inspired by (\cite{PT}, Lemma
$8$).

\begin{lemma}\label{main1lemma2}
If $r$, the number of the fixed points, is no more than $n$, then
$A_{t}=0$ for all $t=1,\cdots,l$.
\end{lemma}
\begin{proof}
For each $i=0,1,\cdots, r-1,$ take
$$f_{i}(x_{1},\cdots,x_{mn})=\big[\prod_{t=1}^{u}\big(\sum_{1\leq
j_{1}<\cdots<j_{\lambda_{t}}\leq mn}x_{j_{1}}x_{j_{2}}\cdots
x_{j_{\lambda_{t}}}\big)\big]^{i}.$$

Here $f_{0}(x_{1},\cdots,x_{mn})=1$. Note that the degree of
$f_{i}(x_{1},\cdots,x_{mn})$ is $mi$ as the weight of $\lambda$ is
$m$, and therefore is less than $mn$ as $r\leq n$ by assumption.
Replacing $f(x_{1},\cdots,x_{mn})$ in Theorem \ref{BRF1} by
$f_{i}(x_{1},\cdots,x_{mn})$ for $i=0,1,\cdots, r-1,$ we have
\begin{eqnarray}\label{matrix}
\left\{\begin{array}{l}
A_{1}+A_{2}+\cdots+A_{l}=0\\
s_{1}A_{1}+s_{2}A_{2}+\cdots+s_{l}A_{l}=0\\
\vdots\\
(s_{1})^{r-1}A_{1}+(s_{2})^{r-1}A_{2}+\cdots+(s_{l})^{r-1}A_{l}=0
\end{array}
\right. \end{eqnarray} Note that $l$ is no more than $r$ by
(\ref{Chern map}) and $s_{1},\cdots,s_{l}$ are mutually distinct,
which means the coefficient matrix of the first $l$ lines of
(\ref{matrix}) is the nonsingular Vandermonde matrix. Hence the only
possibility is
$$A_{1}=\cdots=A_{l}=0.$$
\end{proof}
Combining Lemma \ref{main1lemma1} with Lemma \ref{main1lemma2} will
lead to the proof of the first part of Theorem \ref{maintheorem1}.
The proof of the second part is similar and so we omit it.
\begin{remark}
It is not surprise that Bott's residue formula we used here is
similar to the Atiyah-Bott-Berline-Vergne localization formula used
in \cite{PT}. In fact it turns out that Bott's residue formula can
be put into the framework of the equivariant cohomology theory
(\cite{AB}, \cite{BV}). But Bott's original formula is more suitable
for our purpose. Note that our sufficient condition (vanishing of
some characteristic number) guaranteeing an explicit lower bound of
the number of fixed points relies \emph{only} on the manifold itself
while the sufficient condition in (\cite{PT}, Theorem 1) relies on
the data near the fixed points of the action. But it seems to us
that our result is independent of that of Pelayo-Tolman.
\end{remark}
\subsection{Proof of Theorem \ref{maintheorem2}}
In this section we assume that $N^{2n}$ has a circle action with
isolated fixed points and keep the notations of Section $2.1.2$ in
mind.

The following proposition says, if the action is semi-free, then
$[N^{2n}]$ is at most a torsion element in the oriented cobordism
ring
$\Omega_{\ast}^{SO}$.

\begin{proposition}[Pantilie-Wood]\label{vanish prop}
Suppose $N^{2n}$ has a semi-free $S^{1}$-action with isolated fixed
points. Then all the Pontrjagin numbers of $N^{2n}$ vanish.
Equivalently,
$$[N^{2n}]=0\in\Omega_{\ast}^{SO}\otimes\mathbb{Q}.$$
\end{proposition}
The proof of (\cite{PW}, Theorem 1.1) also uses Bott residue
formula, in the language of differential geometry. Here we give a
quite direct topological proof, although the essential is the same.
\begin{proof}
When $n$ is odd, this proposition obviously holds for dimensional
reason.

Suppose $n$ is even, say $2q$. As noted in Section $2.1.2$, in each
fixed point $P_{i}$, the weights $k_{1}^{(i)},\cdots,k_{n}^{(i)}$
are unique up to even number of sign changes. Since the action is
semi-free, all these $k_{j}^{(i)}$ are $\pm 1$. Let $\rho_{0}$
(resp. $\rho_{1}$) be the number of fixed points with even (resp.
odd) $-1$.

Take $f=1$ in (\ref{BRF2}), we have \be\label{even number formula}
\rho_{0}-\rho_{1}=0,\ee which means the number of the fixed points
are even.

Let $\lambda=(\lambda_{1},\cdots,\lambda_{l})$ be a partition of
$q$. According to (\ref{BRF2}), the corresponding Pontrjagin number
$p_{\lambda}[N]=p_{\lambda_{1}}\cdots p_{\lambda_{l}}[N]$ equals to
$$\binom{2q}{\lambda_{1}}\cdots\binom{2q}{\lambda_{l}}( \rho_{0}-\rho_{1})=0.$$
This completes the proof of this proposition.
\end{proof}

In their famous book \cite{CF}, Conner and Floyd have developed
several bordism techniques and found many interesting applications
in manifolds with group actions. In \cite{KU}, by using the
techniques in \cite{CF}, Kawakubo and Uchida proved several
interesting results related to the signature of manifolds admitting
semi-free circle actions. Among other things, they proved a result
(\cite{KU}, Lemma 2.2), which localizes the cobordism class of the
global manifold to those of the connected components of the fixed
point set. This result is purely constructive and the key ideas are
taken from \cite{CF}. For more details, please consult the original
paper \cite{KU}. Here, for our purpose, from (\cite{KU}, Lemma 2.2)
we have \be\label{cobordism
formula}[N^{2n}]=\sum_{P_{i}}[\mathbb{C}P^{n}\big|_{P_{i}}]\in\Omega_{\ast}^{SO},\ee
where $\mathbb{C}P^{n}\big|_{P_{i}}$ is the $n$-dimensional complex
projective space associated to the fixed point $P_{i}$ and given a
suitable orientation.

When $n$ is odd, $[\mathbb{C}P^{n}]=0\in\Omega_{\ast}^{SO}$ (cf.
\cite{HBJ}, p.1) and therefore $[N^{2n}]=0$.

When $n$ is even, say $2q$, Proposition \ref{vanish prop} and
(\ref{cobordism formula}) imply \be\label{cobordism formula
2}\sum_{P_{i}}[\mathbb{C}P^{2q}\big|_{P_{i}}]=0\in\Omega_{\ast}^{SO}\otimes\mathbb{Q}.\ee
It is well-known that $[\mathbb{C}P^{2q}]$ is \emph{not} a torsion
element. From (\ref{even number formula}) we have known the number
of the fixed points $\{P_{i}\}$ are even. Hence the only possibility
that (\ref{cobordism formula 2}) holds is that half of the
orientations of such $\mathbb{C}P^{2q}\big|_{P_{i}}$ are canonical
and half are opposite to the canonical orientation, which means the
right-hand side, and therefore the left-hand side of (\ref{cobordism
formula}) are zero. This completes the proof of Theorem
\ref{maintheorem2}.

\begin{remark}
In a recent paper \cite{LL}, we have generalized some results of
\cite{KU} and explored some vanishing results by using the rigidity
of elliptic genus.
\end{remark}
\subsection{Proof of Theorem \ref{maintheorem3}}
In this
 subsection, $(M^{2n},J)$ is an almost-complex manifold with a semi-free circle action with isolated fixed points.
 Let $\rho_{t}$ be the number of fixed points of the circle action with exactly $t$ negative weights. In fact these
 $\rho_{t}$ are all related to each other (\cite{TW}, Lemma 3.1)
 $$\rho_{t}=\rho_{0}\cdot\binom{n}{t},\qquad 0\leq t\leq n.$$
 This fact can also be derived from a rigidity result (cf. \cite{Li}, Theorem
 3.2). This fact means the isolated fixed point set is nonempty
 if and only if $\rho_{0}>0$.

The following lemma shows that the first Chern class of $(M^{2n},J)$
is
nonzero.

\begin{lemma}\label{c_{1}c_{n-1}}
The Chern number $c_{1}c_{n-1}[M]$ is equal to $\rho_{0}\cdot
n\cdot2^{n}$. In particular, if the isolated fixed points set is
nonempty, then $c_{1}(M)$ is nonzero.
\end{lemma}
\begin{proof}
In each fixed point $P_{i}$, the $n$ weights
$k_{1}^{(i)},\cdots,k_{n}^{(i)}$ are all $\pm 1$. If the number of
$-1$ among $k_{1}^{(i)},\cdots,k_{n}^{(i)}$ is $t$, then it is easy
to check
$$\frac{e_{1}(k_{1}^{(i)},\cdots,k_{n}^{(i)})e_{n-1}(k_{1}^{(i)},\cdots,k_{n}^{(i)})}{\prod_{j=1}^{n}k_{j}^{(i)}}=(n-2t)^{2}.$$
By Bott's residue formula (\ref{BRF1}) we
have
$$c_{1}c_{n-1}[M]=\sum_{t=0}^{n}\rho_{t}(n-2t)^{2}=\rho_{0}\sum_{t=0}^{n}\binom{n}{t}(n-2t)^{2}=\rho_{0}\cdot n\cdot 2^{n}.$$
\end{proof}

Now we can prove our last main result, Theorem
\ref{maintheorem3}.
\begin{proof}
Since $c_{1}(M)\neq 0$, we can assume $c_{1}(M)=d\cdot x$, where $d$
is a positive integer and $x\in H^{2}(M;\mathbb{Z})$ is a primitive
element. It suffices to show, if $d>1$, then $d$ must be
$2$.

Using Proposition \ref{rigidity proposition} we have
\be\begin{split} 0&\equiv\sum_{t=0}^{n}\rho_{t}g^{\frac{n-2t}{d}}\frac{(-g)^{t}}{(1-g)^{n}}\\
&=\rho_{0}\frac{g^{\frac{n}{d}}}{(1-g)^{n}}\sum^{n}_{t=0}\binom{n}{t}(-1)^{t}g^{\frac{(d-2)t}{d}}\\
&=\rho_{0}\frac{g^{\frac{n}{d}}}{(1-g)^{n}}(1-g^{\frac{d-2}{d}})^{n}.\end{split}\nonumber\ee
If the isolated fixed point set is nonempty, then $\rho_{0}>0$. In
this case, the last expression is identically zero if and only if
$d=2$.
\end{proof}
\bibliographystyle{amsplain}
{\bf Acknowledgements.}~This paper was conducted when the first
author visited Center of Mathematical Sciences, Zhejiang University.
The author thanks the Center for its hospitality. The first author
also thanks Zhi L\"{u} for many fruitful discussions on the related
topics of this paper. The authors thank the referee for his/her
careful reading of the earlier version of this paper and many
fruitful comments and suggestions, which improve the quality of this
paper.

\end{document}